\documentclass{amsart}
\usepackage{amssymb}
\usepackage{amsfonts}

\newtheorem{theorem}{Theorem}
\theoremstyle{plain}

\newtheorem{corollary}{Corollary}

\newtheorem{definition}{Definition}

\newtheorem{lemma}{Lemma}

\newtheorem{remark}{Remark}

\numberwithin{equation}{section}
\input{tcilatex}

\begin{document}
\title[ON\ THE\ CO-ORDINATED CONVEX\ FUNCTIONS]{ON\ THE\ CO-ORDINATED
CONVEX\ FUNCTIONS}
\author{M. Emin \"{O}zdemir$^{\blacksquare }$}
\address{$^{\blacksquare }$Atat\"{u}rk University, K.K. Education Faculty,
Department of Mathematics, 25240, Campus, Erzurum, Turkey}
\email{emos@atauni.edu.tr}
\author{\c{C}etin Y\i ld\i z$^{\blacksquare \spadesuit }$}
\email{yildizc@atauni.edu.tr}
\author{Ahmet Ocak Akdemir$^{\bigstar }$}
\address{$^{\bigstar }$A\u{g}r\i\ \.{I}brahim \c{C}e\c{c}en University,
Faculty of Science and Letters, Department of Mathematics, 04100, A\u{g}r\i
, Turkey}
\email{ahmetakdemir@agri.edu.tr}
\date{April, 2011}
\subjclass[2000]{ Primary 26D15}
\keywords{Hadamard's inequality, co-ordinates, convexity, H\"{o}lder's
inequality, Power mean inequality.\\
$^{\spadesuit }$corresponding author}

\begin{abstract}
In this paper we established new integral inequalities which are more
general results for co-ordinated convex functions on the co-ordinates by
using some classical inequalities.
\end{abstract}

\maketitle

\section{INTRODUCTION}

Let $f:I\subseteq 
\mathbb{R}
\rightarrow 
\mathbb{R}
$ be a convex function defined on the interval $I$ of real numbers and $a<b.$
The following double inequality%
\begin{equation*}
f\left( \frac{a+b}{2}\right) \leq \frac{1}{b-a}\dint\limits_{a}^{b}f(x)dx%
\leq \frac{f(a)+f(b)}{2}
\end{equation*}

is well known in the literature as Hadamard's inequality. Both inequalities
hold in the reversed direction if $f$ is concave.

In \cite{SS}, Dragomir defined convex functions on the co-ordinates as
following;

\begin{definition}
Let us consider the bidimensional interval $\Delta =[a,b]\times \lbrack c,d]$
in $%
\mathbb{R}
^{2}$ with $a<b,$ $c<d.$ A function $f:\Delta \rightarrow 
\mathbb{R}
$ will be called convex on the co-ordinates if the partial mappings $%
f_{y}:[a,b]\rightarrow 
\mathbb{R}
,$ $f_{y}(u)=f(u,y)$ and $f_{x}:[c,d]\rightarrow 
\mathbb{R}
,$ $f_{x}(v)=f(x,v)$ are convex where defined for all $y\in \lbrack c,d]$
and $x\in \lbrack a,b].$ Recall that the mapping $f:\Delta \rightarrow 
\mathbb{R}
$ is convex on $\Delta $ if the following inequality holds, 
\begin{equation*}
f(\lambda x+(1-\lambda )z,\lambda y+(1-\lambda )w)\leq \lambda
f(x,y)+(1-\lambda )f(z,w)
\end{equation*}%
for all $(x,y),(z,w)\in \Delta $ and $\lambda \in \lbrack 0,1].$
\end{definition}

In \cite{SS}, Dragomir established the following inequalities of Hadamard's
type for co-ordinated convex functions on a rectangle from the plane $%
\mathbb{R}
^{2}.$

\begin{theorem}
Suppose that $f:\Delta =[a,b]\times \lbrack c,d]\rightarrow 
\mathbb{R}
$ is convex on the co-ordinates on $\Delta $. Then one has the inequalities;%
\begin{eqnarray}
&&\ f(\frac{a+b}{2},\frac{c+d}{2})  \label{1.1} \\
&\leq &\frac{1}{2}\left[ \frac{1}{b-a}\int_{a}^{b}f(x,\frac{c+d}{2})dx+\frac{%
1}{d-c}\int_{c}^{d}f(\frac{a+b}{2},y)dy\right]  \notag \\
&\leq &\frac{1}{(b-a)(d-c)}\int_{a}^{b}\int_{c}^{d}f(x,y)dxdy  \notag \\
&\leq &\frac{1}{4}\left[ \frac{1}{(b-a)}\int_{a}^{b}f(x,c)dx+\frac{1}{(b-a)}%
\int_{a}^{b}f(x,d)dx\right.  \notag \\
&&\left. +\frac{1}{(d-c)}\int_{c}^{d}f(a,y)dy+\frac{1}{(d-c)}%
\int_{c}^{d}f(b,y)dy\right]  \notag \\
&\leq &\frac{f(a,c)+f(a,d)+f(b,c)+f(b,d)}{4}.  \notag
\end{eqnarray}%
The above inequalities are sharp.
\end{theorem}

In \cite{BAK}, Bakula and Pe\v{c}ari\'{c} established several Jensen type
inequalities for co-ordinated convex functions and in \cite{HTS}, Hwang 
\textit{et al.} gave a mapping $F$, discussed some properties of this
mapping and proved some Hadamard-type inequalities for Lipschizian mapping
in two variables. In \cite{OZ}, \"{O}zdemir \textit{et al. }established%
\textit{\ }new Hadamard-type inequalities\textit{\ }for co-ordinated $m-$%
convex and $\left( \alpha ,m\right) -$convex functions. On all of these, in 
\cite{OZ2}, the authors proved some Hadamard-type inequalities for
co-ordinated convex functions as followings;

\begin{theorem}
Let $f:\Delta \subset 
\mathbb{R}
^{2}\rightarrow 
\mathbb{R}
$ be a partial differentiable mapping on $\Delta :=[a,b]\times \lbrack c,d]$
in $%
\mathbb{R}
^{2}$ with $a<b$ and $c<d.$ If $\left\vert \frac{\partial ^{2}f}{\partial
t\partial s}\right\vert $ is a convex function on the co-ordinates on $%
\Delta ,$ then one has the inequalities:%
\begin{eqnarray}
&&\left\vert \frac{f(a,c)+f(a,d)+f(b,c)+f(b,d)}{4}\right.  \label{1.3} \\
&&\left. \frac{1}{(b-a)(d-c)}\int_{a}^{b}\int_{c}^{d}f(x,y)dxdy-A\right\vert
\notag \\
&\leq &\frac{(b-a)(d-c)}{16}\left( \frac{\left\vert \frac{\partial ^{2}f}{%
\partial t\partial s}\right\vert (a,c)+\left\vert \frac{\partial ^{2}f}{%
\partial t\partial s}\right\vert (a,d)+\left\vert \frac{\partial ^{2}f}{%
\partial t\partial s}\right\vert (b,c)+\left\vert \frac{\partial ^{2}f}{%
\partial t\partial s}\right\vert (b,d)}{4}\right)  \notag
\end{eqnarray}%
where%
\begin{equation*}
A=\frac{1}{2}\left[ \frac{1}{(b-a)}\int_{a}^{b}\left[ f(x,c)+f(x,d)\right]
dx+\frac{1}{(d-c)}\int_{c}^{d}\left[ f(a,y)dy+f(b,y)\right] dy\right] .
\end{equation*}
\end{theorem}

\begin{theorem}
Let $f:\Delta \subset 
\mathbb{R}
^{2}\rightarrow 
\mathbb{R}
$ be a partial differentiable mapping on $\Delta :=[a,b]\times \lbrack c,d]$
in $%
\mathbb{R}
^{2}$ with $a<b$ and $c<d.$ If $\left\vert \frac{\partial ^{2}f}{\partial
t\partial s}\right\vert ^{q},$ $q>1,$ is a convex function on the
co-ordinates on $\Delta ,$ then one has the inequalities:%
\begin{eqnarray}
&&\left\vert \frac{f(a,c)+f(a,d)+f(b,c)+f(b,d)}{4}\right.  \label{1.4} \\
&&\left. \frac{1}{(b-a)(d-c)}\int_{a}^{b}\int_{c}^{d}f(x,y)dxdy-A\right\vert
\notag \\
&\leq &\frac{(b-a)(d-c)}{4\left( p+1\right) ^{\frac{2}{p}}}\left( \frac{%
\left\vert \frac{\partial ^{2}f}{\partial t\partial s}\right\vert
^{q}(a,c)+\left\vert \frac{\partial ^{2}f}{\partial t\partial s}\right\vert
^{q}(a,d)+\left\vert \frac{\partial ^{2}f}{\partial t\partial s}\right\vert
^{q}(b,c)+\left\vert \frac{\partial ^{2}f}{\partial t\partial s}\right\vert
^{q}(b,d)}{4}\right) ^{\frac{1}{q}}  \notag
\end{eqnarray}%
where%
\begin{equation*}
A=\frac{1}{2}\left[ \frac{1}{(b-a)}\int_{a}^{b}\left[ f(x,c)+f(x,d)\right]
dx+\frac{1}{(d-c)}\int_{c}^{d}\left[ f(a,y)dy+f(b,y)\right] dy\right]
\end{equation*}%
and $\frac{1}{p}+\frac{1}{q}=1.$
\end{theorem}

\begin{theorem}
Let $f:\Delta \subset 
\mathbb{R}
^{2}\rightarrow 
\mathbb{R}
$ be a partial differentiable mapping on $\Delta :=[a,b]\times \lbrack c,d]$
in $%
\mathbb{R}
^{2}$ with $a<b$ and $c<d.$ If $\left\vert \frac{\partial ^{2}f}{\partial
t\partial s}\right\vert ^{q},$ $q\geq 1,$ is a convex function on the
co-ordinates on $\Delta ,$ then one has the inequalities:%
\begin{eqnarray}
&&\left\vert \frac{f(a,c)+f(a,d)+f(b,c)+f(b,d)}{4}\right.  \label{1.5} \\
&&\left. \frac{1}{(b-a)(d-c)}\int_{a}^{b}\int_{c}^{d}f(x,y)dxdy-A\right\vert
\notag \\
&\leq &\frac{(b-a)(d-c)}{16}\left( \frac{\left\vert \frac{\partial ^{2}f}{%
\partial t\partial s}\right\vert ^{q}(a,c)+\left\vert \frac{\partial ^{2}f}{%
\partial t\partial s}\right\vert ^{q}(a,d)+\left\vert \frac{\partial ^{2}f}{%
\partial t\partial s}\right\vert ^{q}(b,c)+\left\vert \frac{\partial ^{2}f}{%
\partial t\partial s}\right\vert ^{q}(b,d)}{4}\right) ^{\frac{1}{q}}  \notag
\end{eqnarray}%
where%
\begin{equation*}
A=\frac{1}{2}\left[ \frac{1}{(b-a)}\int_{a}^{b}\left[ f(x,c)+f(x,d)\right]
dx+\frac{1}{(d-c)}\int_{c}^{d}\left[ f(a,y)dy+f(b,y)\right] dy\right] .
\end{equation*}
\end{theorem}

In \cite{OZ3}, authors proved following inequalities for co-ordinated convex
functions.

\begin{theorem}
Let $f:\Delta =\left[ a,b\right] \times \left[ c,d\right] \rightarrow 
\mathbb{R}
$ be a partial differentiable mapping on $\Delta =\left[ a,b\right] \times %
\left[ c,d\right] .$ If $\left\vert \frac{\partial ^{2}f}{\partial t\partial
s}\right\vert $ is a convex function on the co-ordinates on $\Delta ,$ then
the following inequality holds;%
\begin{eqnarray}
&&\left\vert f\left( \frac{a+b}{2},\frac{c+d}{2}\right) \right.  \label{1.6}
\\
&&-\frac{1}{\left( d-c\right) }\int_{c}^{d}f\left( \frac{a+b}{2},y\right) dy-%
\frac{1}{\left( b-a\right) }\int_{a}^{b}f\left( x,\frac{c+d}{2}\right) dx 
\notag \\
&&\left. +\frac{1}{\left( b-a\right) \left( d-c\right) }\int_{a}^{b}%
\int_{c}^{d}f\left( x,y\right) dydx\right\vert  \notag \\
&\leq &\frac{\left( b-a\right) \left( d-c\right) }{64}\left[ \left\vert 
\frac{\partial ^{2}f}{\partial t\partial s}\left( a,c\right) \right\vert
+\left\vert \frac{\partial ^{2}f}{\partial t\partial s}\left( b,c\right)
\right\vert +\left\vert \frac{\partial ^{2}f}{\partial t\partial s}\left(
a,d\right) \right\vert +\left\vert \frac{\partial ^{2}f}{\partial t\partial s%
}\left( b,d\right) \right\vert \right] .  \notag
\end{eqnarray}
\end{theorem}

\begin{theorem}
Let $f:\Delta =\left[ a,b\right] \times \left[ c,d\right] \rightarrow 
\mathbb{R}
$ be a partial differentiable mapping on $\Delta =\left[ a,b\right] \times %
\left[ c,d\right] .$ If $\left\vert \frac{\partial ^{2}f}{\partial t\partial
s}\right\vert ^{q},$ $q>1,$ is a convex function on the co-ordinates on $%
\Delta ,$ then the following inequality holds;%
\begin{eqnarray}
&&\left\vert f\left( \frac{a+b}{2},\frac{c+d}{2}\right) +\frac{1}{\left(
b-a\right) \left( d-c\right) }\int_{a}^{b}\int_{c}^{d}f\left( x,y\right)
dydx\right.  \label{1.7} \\
&&\left. -\frac{1}{\left( d-c\right) }\int_{c}^{d}f\left( \frac{a+b}{2}%
,y\right) dy-\frac{1}{\left( b-a\right) }\int_{a}^{b}f\left( x,\frac{c+d}{2}%
\right) dx\right\vert  \notag \\
&\leq &\frac{\left( b-a\right) \left( d-c\right) }{4\left( p+1\right) ^{%
\frac{2}{p}}}  \notag \\
&&\times \left( \frac{\left\vert \frac{\partial ^{2}f}{\partial t\partial s}%
\left( a,c\right) \right\vert ^{q}+\left\vert \frac{\partial ^{2}f}{\partial
t\partial s}\left( b,c\right) \right\vert ^{q}+\left\vert \frac{\partial
^{2}f}{\partial t\partial s}\left( a,d\right) \right\vert ^{q}+\left\vert 
\frac{\partial ^{2}f}{\partial t\partial s}\left( b,d\right) \right\vert ^{q}%
}{4}\right) ^{\frac{1}{q}}.  \notag
\end{eqnarray}
\end{theorem}

\begin{theorem}
Let $f:\Delta =\left[ a,b\right] \times \left[ c,d\right] \rightarrow 
\mathbb{R}
$ be a partial differentiable mapping on $\Delta =\left[ a,b\right] \times %
\left[ c,d\right] .$ If $\left\vert \frac{\partial ^{2}f}{\partial t\partial
s}\right\vert ^{q},$ $q\geq 1,$ is a convex function on the co-ordinates on $%
\Delta ,$ then the following inequality holds;%
\begin{eqnarray}
&&\left\vert f\left( \frac{a+b}{2},\frac{c+d}{2}\right) +\frac{1}{\left(
b-a\right) \left( d-c\right) }\int_{a}^{b}\int_{c}^{d}f\left( x,y\right)
dydx\right.  \label{1.8} \\
&&\left. -\frac{1}{\left( d-c\right) }\int_{c}^{d}f\left( \frac{a+b}{2}%
,y\right) dy-\frac{1}{\left( b-a\right) }\int_{a}^{b}f\left( x,\frac{c+d}{2}%
\right) dx\right\vert  \notag \\
&\leq &\frac{\left( b-a\right) \left( d-c\right) }{16}  \notag \\
&&\times \left( \frac{\left\vert \frac{\partial ^{2}f}{\partial t\partial s}%
\left( a,c\right) \right\vert ^{q}+\left\vert \frac{\partial ^{2}f}{\partial
t\partial s}\left( b,c\right) \right\vert ^{q}+\left\vert \frac{\partial
^{2}f}{\partial t\partial s}\left( a,d\right) \right\vert ^{q}+\left\vert 
\frac{\partial ^{2}f}{\partial t\partial s}\left( b,d\right) \right\vert ^{q}%
}{4}\right) ^{\frac{1}{q}}.  \notag
\end{eqnarray}
\end{theorem}

In \cite{OZ4}, \"{O}zdemir \textit{et al. }proved the following Theorem
which is involving an inequality of Simpson's type;

\begin{theorem}
Let $f:\Delta \subset 
\mathbb{R}
^{2}\rightarrow 
\mathbb{R}
$ be a partial differentiable mapping on $\Delta =\left[ a,b\right] \times %
\left[ c,d\right] .$ If $\frac{\partial ^{2}f}{\partial t\partial s}$ is a
convex function on the co-ordinates on $\Delta ,$ then the following
inequality holds:%
\begin{eqnarray}
&&\left\vert \frac{f\left( a,\frac{c+d}{2}\right) +f\left( b,\frac{c+d}{2}%
\right) +4f\left( \frac{a+b}{2},\frac{c+d}{2}\right) +f\left( \frac{a+b}{2}%
,c\right) +f\left( \frac{a+b}{2},d\right) }{9}\right.  \notag \\
&&+\frac{f\left( a,c\right) +f\left( b,c\right) +f\left( a,d\right) +f\left(
b,d\right) }{36}\left. +\frac{1}{\left( b-a\right) \left( d-c\right) }%
\int_{a}^{b}\int_{c}^{d}f\left( x,y\right) dydx-A\right\vert  \notag \\
&\leq &\frac{25\left( b-a\right) \left( d-c\right) }{72}  \label{1.9} \\
&&\times \left( \frac{\left\vert \frac{\partial ^{2}f}{\partial t\partial s}%
\left( a,c\right) \right\vert +\left\vert \frac{\partial ^{2}f}{\partial
t\partial s}\left( a,d\right) \right\vert +\left\vert \frac{\partial ^{2}f}{%
\partial t\partial s}\left( b,c\right) \right\vert +\left\vert \frac{%
\partial ^{2}f}{\partial t\partial s}\left( b,d\right) \right\vert }{72}%
\right)  \notag
\end{eqnarray}%
where%
\begin{eqnarray*}
A &=&\frac{1}{6\left( b-a\right) }\int_{a}^{b}\left[ f\left( x,c\right)
+4f\left( x,\frac{c+d}{2}\right) +f\left( x,d\right) \right] dx \\
&&+\frac{1}{6\left( d-c\right) }\int_{c}^{d}\left[ f\left( a,y\right)
+4f\left( \frac{a+b}{2},y\right) +f\left( b,y\right) \right] dy.
\end{eqnarray*}
\end{theorem}

The main purpose of this paper is to establish a new lemma which give more
general results and different type inequalities for special values of $%
\lambda $ and to prove several inequalities.

\section{MAIN\ RESULTS}

In order to prove our main theorems we need the following lemma:

\begin{lemma}
Let $f:\Delta \subset 
\mathbb{R}
^{2}\rightarrow 
\mathbb{R}
$ be a differentiable function on $\Delta $ where $a<b,$ $c<d$ and $\lambda
\in \left[ 0,1\right] $, if $\frac{\partial ^{2}f}{\partial x\partial y}\in
L_{1}\left( \Delta \right) $, then the following equality holds:%
\begin{eqnarray*}
&&\left( 1-\lambda \right) ^{2}f\left( \frac{a+b}{2},\frac{c+d}{2}\right)
+\lambda ^{2}\frac{f\left( a,c\right) +f\left( a,d\right) +f\left(
b,c\right) +f\left( b,d\right) }{4} \\
&&+\frac{\lambda \left( 1-\lambda \right) }{2}\left[ f\left( \frac{a+b}{2}%
,c\right) +f\left( \frac{a+b}{2},d\right) +f\left( a,\frac{c+d}{2}\right)
+f\left( b,\frac{c+d}{2}\right) \right] \\
&&-\left( 1-\lambda \right) \frac{1}{d-c}\int_{c}^{d}f(\frac{a+b}{2}%
,y)dy-\left( 1-\lambda \right) \frac{1}{b-a}\int_{a}^{b}f(x,\frac{c+d}{2})dx
\\
&&-\lambda \frac{1}{2\left( b-a\right) }\int_{a}^{b}\left[ f(x,d)+f(x,c)%
\right] dx-\lambda \frac{1}{2\left( d-c\right) }\int_{c}^{d}\left[
f(a,y)+f(b,y)\right] dy \\
&&+\frac{1}{(b-a)(d-c)}\int_{a}^{b}\int_{c}^{d}f(x,y)dxdy \\
&=&\frac{1}{(b-a)(d-c)}\int_{a}^{b}\int_{c}^{d}K(x)M(y)\frac{\partial ^{2}f}{%
\partial x\partial y}(x,y)dydx
\end{eqnarray*}%
where%
\begin{equation*}
K(x)=\left\{ 
\begin{array}{c}
x-\left( a+\lambda \frac{b-a}{2}\right) ,\text{ \ \ }x\in \left[ a,\frac{a+b%
}{2}\right] \\ 
\\ 
x-\left( b-\lambda \frac{b-a}{2}\right) ,\text{ \ \ }x\in \left[ \frac{a+b}{2%
},b\right]%
\end{array}%
\right.
\end{equation*}%
and%
\begin{equation*}
M(y)=\left\{ 
\begin{array}{c}
y-\left( c+\lambda \frac{d-c}{2}\right) ,\text{ \ \ }y\in \left[ c,\frac{c+d%
}{2}\right] \\ 
\\ 
y-\left( d-\lambda \frac{d-c}{2}\right) ,\text{ \ \ }y\in \left[ \frac{c+d}{2%
},d\right]%
\end{array}%
\right. .
\end{equation*}
\end{lemma}

\begin{proof}
Integration by parts, we get%
\begin{eqnarray*}
&&\int_{a}^{b}\int_{c}^{d}K(x)M(y)\frac{\partial ^{2}f}{\partial x\partial y}%
(x,y)dydx \\
&=&\int_{a}^{b}K(x)\left[ \int_{c}^{\frac{c+d}{2}}\left( y-\left( c+\lambda 
\frac{d-c}{2}\right) \right) \frac{\partial ^{2}f}{\partial x\partial y}%
(x,y)dy\right. \\
&&\left. +\int_{\frac{c+d}{2}}^{d}\left( y-\left( d-\lambda \frac{d-c}{2}%
\right) \right) \frac{\partial ^{2}f}{\partial x\partial y}(x,y)dy\right] dx
\\
&=&\int_{a}^{b}K(x)\left[ \left. \left( y-\left( c+\lambda \frac{d-c}{2}%
\right) \right) \frac{\partial f}{\partial x}(x,y)\right\vert _{c}^{\frac{c+d%
}{2}}-\int_{c}^{\frac{c+d}{2}}\frac{\partial f}{\partial x}(x,y)dy\right. \\
&&\left. \left. +\left( y-\left( d-\lambda \frac{d-c}{2}\right) \right) 
\frac{\partial f}{\partial x}(x,y)\right\vert _{\frac{c+d}{2}}^{d}-\int_{%
\frac{c+d}{2}}^{d}\frac{\partial f}{\partial x}(x,y)dy\right] dx \\
&=&\int_{a}^{b}K(x)\left[ \left( 1-\lambda \right) \left( d-c\right) \frac{%
\partial f}{\partial x}\left( x,\frac{c+d}{2}\right) \right. \\
&&\left. +\left( \lambda \frac{d-c}{2}\right) \left( \frac{\partial f}{%
\partial x}\left( x,c\right) +\frac{\partial f}{\partial x}\left( x,d\right)
\right) -\int_{c}^{d}\frac{\partial f}{\partial x}(x,y)dy\right] dx.
\end{eqnarray*}%
Integrating by parts again, we obtain%
\begin{eqnarray*}
&&\int_{a}^{b}\int_{c}^{d}K(x)M(y)\frac{\partial ^{2}f}{\partial x\partial y}%
(x,y)dydx \\
&=&\left( 1-\lambda \right) ^{2}(b-a)(d-c)\left[ f\left( \frac{a+b}{2},\frac{%
c+d}{2}\right) \right] \\
&&+\lambda ^{2}(b-a)(d-c)\left[ \frac{f\left( a,c\right) +f\left( a,d\right)
+f\left( b,c\right) +f\left( b,d\right) }{4}\right] \\
&&+\frac{\lambda \left( 1-\lambda \right) (b-a)(d-c)}{2}\left[ f\left( \frac{%
a+b}{2},c\right) +f\left( \frac{a+b}{2},d\right) \right. \\
&&\left. +f\left( a,\frac{c+d}{2}\right) +f\left( b,\frac{c+d}{2}\right) %
\right] -\left( 1-\lambda \right) (b-a)\int_{c}^{d}f(\frac{a+b}{2},y)dy \\
&&-\left( 1-\lambda \right) (d-c)\int_{a}^{b}f(x,\frac{c+d}{2})dx-\lambda 
\frac{(d-c)}{2}\int_{a}^{b}\left[ f(x,d)+f(x,c)\right] dx \\
&&-\lambda \frac{(b-a)}{2}\int_{c}^{d}\left[ f(a,y)+f(b,y)\right]
dy+\int_{a}^{b}\int_{c}^{d}f(x,y)dxdy.
\end{eqnarray*}%
Dividing both sides of the above equality by $(b-a)(d-c),$ we have the
required result.
\end{proof}

\begin{theorem}
Let $f:\Delta =[a,b]\times \lbrack c,d]\rightarrow 
\mathbb{R}
$ be a differentiable function on $\Delta $. If $\left\vert \frac{\partial
^{2}f}{\partial x\partial y}\right\vert $ is convex function on the
co-ordinates on $\Delta ,$ then one has the inequality:%
\begin{eqnarray*}
&&\left\vert \left( 1-\lambda \right) ^{2}f\left( \frac{a+b}{2},\frac{c+d}{2}%
\right) +\lambda ^{2}\frac{f\left( a,c\right) +f\left( a,d\right) +f\left(
b,c\right) +f\left( b,d\right) }{4}\right. \\
&&+\frac{\lambda \left( 1-\lambda \right) }{2}\left[ f\left( \frac{a+b}{2}%
,c\right) +f\left( \frac{a+b}{2},d\right) +f\left( a,\frac{c+d}{2}\right)
+f\left( b,\frac{c+d}{2}\right) \right] \\
&&-\left( 1-\lambda \right) \frac{1}{d-c}\int_{c}^{d}f(\frac{a+b}{2}%
,y)dy-\left( 1-\lambda \right) \frac{1}{b-a}\int_{a}^{b}f(x,\frac{c+d}{2})dx
\\
&&-\lambda \frac{1}{2\left( b-a\right) }\int_{a}^{b}\left[ f(x,d)+f(x,c)%
\right] dx-\lambda \frac{1}{2\left( d-c\right) }\int_{c}^{d}\left[
f(a,y)+f(b,y)\right] dy \\
&&\left. +\frac{1}{(b-a)(d-c)}\int_{a}^{b}\int_{c}^{d}f(x,y)dxdy\right\vert
\\
&\leq &\frac{(b-a)(d-c)}{16}\left[ 2\lambda ^{2}-2\lambda +1\right] ^{2} \\
&&\times \left( \frac{\left\vert \frac{\partial ^{2}f}{\partial x\partial y}%
\right\vert (a,c)+\left\vert \frac{\partial ^{2}f}{\partial x\partial y}%
\right\vert (a,d)+\left\vert \frac{\partial ^{2}f}{\partial x\partial y}%
\right\vert (b,c)+\left\vert \frac{\partial ^{2}f}{\partial x\partial y}%
\right\vert (b,d)}{4}\right) .
\end{eqnarray*}
\end{theorem}

\begin{proof}
From Lemma 1 and property of the modulus, we can write%
\begin{eqnarray*}
&&\left\vert \left( 1-\lambda \right) ^{2}f\left( \frac{a+b}{2},\frac{c+d}{2}%
\right) +\lambda ^{2}\frac{f\left( a,c\right) +f\left( a,d\right) +f\left(
b,c\right) +f\left( b,d\right) }{4}\right. \\
&&+\frac{\lambda \left( 1-\lambda \right) }{2}\left[ f\left( \frac{a+b}{2}%
,c\right) +f\left( \frac{a+b}{2},d\right) +f\left( a,\frac{c+d}{2}\right)
+f\left( b,\frac{c+d}{2}\right) \right] \\
&&-\left( 1-\lambda \right) \frac{1}{d-c}\int_{c}^{d}f(\frac{a+b}{2}%
,y)dy-\left( 1-\lambda \right) \frac{1}{b-a}\int_{a}^{b}f(x,\frac{c+d}{2})dx
\\
&&-\lambda \frac{1}{2\left( b-a\right) }\int_{a}^{b}\left[ f(x,d)+f(x,c)%
\right] dx-\lambda \frac{1}{2\left( d-c\right) }\int_{c}^{d}\left[
f(a,y)+f(b,y)\right] dy \\
&&\left. +\frac{1}{(b-a)(d-c)}\int_{a}^{b}\int_{c}^{d}f(x,y)dxdy\right\vert
\\
&\leq &\frac{1}{(b-a)(d-c)}\int_{a}^{b}\int_{c}^{d}\left\vert
K(x)M(y)\right\vert \left\vert \frac{\partial ^{2}f}{\partial x\partial y}%
\right\vert (x,y)dydx.
\end{eqnarray*}%
By using the change of variables $y=sd+\left( 1-s\right) c,$ $\left(
d-c\right) ds=dy$, we obatin%
\begin{eqnarray*}
&&\left\vert \left( 1-\lambda \right) ^{2}f\left( \frac{a+b}{2},\frac{c+d}{2}%
\right) +\lambda ^{2}\frac{f\left( a,c\right) +f\left( a,d\right) +f\left(
b,c\right) +f\left( b,d\right) }{4}\right. \\
&&+\frac{\lambda \left( 1-\lambda \right) }{2}\left[ f\left( \frac{a+b}{2}%
,c\right) +f\left( \frac{a+b}{2},d\right) +f\left( a,\frac{c+d}{2}\right)
+f\left( b,\frac{c+d}{2}\right) \right] \\
&&-\left( 1-\lambda \right) \frac{1}{d-c}\int_{c}^{d}f(\frac{a+b}{2}%
,y)dy-\left( 1-\lambda \right) \frac{1}{b-a}\int_{a}^{b}f(x,\frac{c+d}{2})dx
\\
&&-\lambda \frac{1}{2\left( b-a\right) }\int_{a}^{b}\left[ f(x,d)+f(x,c)%
\right] dx-\lambda \frac{1}{2\left( d-c\right) }\int_{c}^{d}\left[
f(a,y)+f(b,y)\right] dy \\
&&\left. +\frac{1}{(b-a)(d-c)}\int_{a}^{b}\int_{c}^{d}f(x,y)dxdy\right\vert
\\
&\leq &\frac{d-c}{b-a}\int_{a}^{b}\left\vert K(x)\right\vert \left\{
\int_{0}^{\frac{\lambda }{2}}\left( \frac{\lambda }{2}-s\right) \left\vert 
\frac{\partial ^{2}f}{\partial x\partial y}(x,sd+\left( 1-s\right)
c)\right\vert ds\right. \\
&&+\int_{\frac{\lambda }{2}}^{\frac{1}{2}}\left( s-\frac{\lambda }{2}\right)
\left\vert \frac{\partial ^{2}f}{\partial x\partial y}(x,sd+\left(
1-s\right) c)\right\vert ds \\
&&+\int_{\frac{1}{2}}^{1-\frac{\lambda }{2}}\left( 1-\frac{\lambda }{2}%
-s\right) \left\vert \frac{\partial ^{2}f}{\partial x\partial y}(x,sd+\left(
1-s\right) c)\right\vert ds \\
&&\left. +\int_{1-\frac{\lambda }{2}}^{1}\left( s-1+\frac{\lambda }{2}%
\right) \left\vert \frac{\partial ^{2}f}{\partial x\partial y}(x,sd+\left(
1-s\right) c)\right\vert ds\right\} dx.
\end{eqnarray*}%
Since $\left\vert \frac{\partial ^{2}f}{\partial x\partial y}\right\vert $
is convex function on the co-ordinates on $\Delta $, we have%
\begin{eqnarray*}
&&\frac{1}{(b-a)(d-c)}\int_{a}^{b}\int_{c}^{d}\left\vert K(x)M(y)\right\vert
\left\vert \frac{\partial ^{2}f}{\partial x\partial y}\right\vert (x,y)dydx
\\
&\leq &\frac{d-c}{b-a}\int_{a}^{b}\left\vert K(x)\right\vert \left\{
\int_{0}^{\frac{\lambda }{2}}s\left( \frac{\lambda }{2}-s\right) \left\vert 
\frac{\partial ^{2}f}{\partial x\partial y}(x,d)\right\vert ds\right. \\
&&+\int_{0}^{\frac{\lambda }{2}}\left( 1-s\right) \left( \frac{\lambda }{2}%
-s\right) \left\vert \frac{\partial ^{2}f}{\partial x\partial y}%
(x,c)\right\vert ds \\
&&+\int_{\frac{\lambda }{2}}^{\frac{1}{2}}s\left( s-\frac{\lambda }{2}%
\right) \left\vert \frac{\partial ^{2}f}{\partial x\partial y}%
(x,d)\right\vert ds+\int_{\frac{\lambda }{2}}^{\frac{1}{2}}\left( 1-s\right)
\left( s-\frac{\lambda }{2}\right) \left\vert \frac{\partial ^{2}f}{\partial
x\partial y}(x,c)\right\vert ds \\
&&+\int_{\frac{1}{2}}^{1-\frac{\lambda }{2}}s\left( 1-\frac{\lambda }{2}%
-s\right) \left\vert \frac{\partial ^{2}f}{\partial x\partial y}%
(x,d)\right\vert ds \\
&&+\int_{\frac{1}{2}}^{1-\frac{\lambda }{2}}\left( 1-s\right) \left( 1-\frac{%
\lambda }{2}-s\right) \left\vert \frac{\partial ^{2}f}{\partial x\partial y}%
(x,c)\right\vert ds \\
&&+\int_{1-\frac{\lambda }{2}}^{1}s\left( s-1+\frac{\lambda }{2}\right)
\left\vert \frac{\partial ^{2}f}{\partial x\partial y}(x,d)\right\vert ds \\
&&\left. +\int_{1-\frac{\lambda }{2}}^{1}\left( 1-s\right) \left( s-1+\frac{%
\lambda }{2}\right) \left\vert \frac{\partial ^{2}f}{\partial x\partial y}%
(x,c)\right\vert ds\right\} dx.
\end{eqnarray*}%
By calculating the above integrals, we obtain%
\begin{eqnarray*}
&&\left\vert \left( 1-\lambda \right) ^{2}f\left( \frac{a+b}{2},\frac{c+d}{2}%
\right) +\lambda ^{2}\frac{f\left( a,c\right) +f\left( a,d\right) +f\left(
b,c\right) +f\left( b,d\right) }{4}\right. \\
&&+\frac{\lambda \left( 1-\lambda \right) }{2}\left[ f\left( \frac{a+b}{2}%
,c\right) +f\left( \frac{a+b}{2},d\right) +f\left( a,\frac{c+d}{2}\right)
+f\left( b,\frac{c+d}{2}\right) \right] \\
&&-\left( 1-\lambda \right) \frac{1}{d-c}\int_{c}^{d}f(\frac{a+b}{2}%
,y)dy-\left( 1-\lambda \right) \frac{1}{b-a}\int_{a}^{b}f(x,\frac{c+d}{2})dx
\\
&&-\lambda \frac{1}{2\left( b-a\right) }\int_{a}^{b}\left[ f(x,d)+f(x,c)%
\right] dx-\lambda \frac{1}{2\left( d-c\right) }\int_{c}^{d}\left[
f(a,y)+f(b,y)\right] dy \\
&&\left. +\frac{1}{(b-a)(d-c)}\int_{a}^{b}\int_{c}^{d}f(x,y)dxdy\right\vert
\\
&\leq &\frac{d-c}{b-a}\int_{a}^{b}\left\vert K(x)\right\vert \left\{ \left[ 
\frac{2\lambda ^{2}-2\lambda +1}{8}\right] \left( \left\vert \frac{\partial
^{2}f}{\partial x\partial y}(x,c)\right\vert +\left\vert \frac{\partial ^{2}f%
}{\partial x\partial y}(x,d)\right\vert \right) \right\} .
\end{eqnarray*}%
By a similar argument for other integrals, by using the change of variable $%
x=tb+(1-t)a,$ $(b-a)dt=dx$ and convexity of $\left\vert \frac{\partial ^{2}f%
}{\partial x\partial y}(x,y)\right\vert $ on the co-ordinates on $\Delta ,$
we deduce the result. Which completes the proof.
\end{proof}

\begin{remark}
If we choose $\lambda =1$ in Theorem 5, we have the inequality (\ref{1.3}).
\end{remark}

\begin{remark}
If we choose $\lambda =0$ in Theorem 5, we have the inequality (\ref{1.6}).
\end{remark}

\begin{remark}
If we choose $\lambda =\frac{1}{3}$ in Theorem 5, we have the inequality (%
\ref{1.9}).
\end{remark}

\begin{theorem}
Let $f:\Delta =[a,b]\times \lbrack c,d]\rightarrow 
\mathbb{R}
$ be a differentiable function on $\Delta $. If $\left\vert \frac{\partial
^{2}f}{\partial x\partial y}\right\vert ^{\frac{p}{p-1}}$ is convex function
on the co-ordinates on $\Delta ,$ then one has the inequality:%
\begin{eqnarray*}
&&\left\vert \left( 1-\lambda \right) ^{2}f\left( \frac{a+b}{2},\frac{c+d}{2}%
\right) +\lambda ^{2}\frac{f\left( a,c\right) +f\left( a,d\right) +f\left(
b,c\right) +f\left( b,d\right) }{4}\right. \\
&&+\frac{\lambda \left( 1-\lambda \right) }{2}\left[ f\left( \frac{a+b}{2}%
,c\right) +f\left( \frac{a+b}{2},d\right) +f\left( a,\frac{c+d}{2}\right)
+f\left( b,\frac{c+d}{2}\right) \right] \\
&&-\left( 1-\lambda \right) \frac{1}{d-c}\int_{c}^{d}f(\frac{a+b}{2}%
,y)dy-\left( 1-\lambda \right) \frac{1}{b-a}\int_{a}^{b}f(x,\frac{c+d}{2})dx
\\
&&-\lambda \frac{1}{2\left( b-a\right) }\int_{a}^{b}\left[ f(x,d)+f(x,c)%
\right] dx-\lambda \frac{1}{2\left( d-c\right) }\int_{c}^{d}\left[
f(a,y)+f(b,y)\right] dy \\
&&\left. +\frac{1}{(b-a)(d-c)}\int_{a}^{b}\int_{c}^{d}f(x,y)dxdy\right\vert
\\
&\leq &\frac{(b-a)(d-c)}{4\left( p+1\right) ^{\frac{2}{p}}}\left[ 2\lambda
^{2}-2\lambda +1\right] ^{2} \\
&&\times \left( \frac{\left\vert \frac{\partial ^{2}f}{\partial x\partial y}%
\right\vert ^{q}(a,c)+\left\vert \frac{\partial ^{2}f}{\partial x\partial y}%
\right\vert ^{q}(a,d)+\left\vert \frac{\partial ^{2}f}{\partial x\partial y}%
\right\vert ^{q}(b,c)+\left\vert \frac{\partial ^{2}f}{\partial x\partial y}%
\right\vert ^{q}(b,d)}{4}\right) ^{\frac{1}{q}}
\end{eqnarray*}%
for $q>1,$ where $q=\frac{p}{p-1}.$
\end{theorem}

\begin{proof}
Let $p>1.$ From Lemma 1 and using the H\"{o}lder inequality for double
integrals, we can write%
\begin{eqnarray*}
&&\left\vert \left( 1-\lambda \right) ^{2}f\left( \frac{a+b}{2},\frac{c+d}{2}%
\right) +\lambda ^{2}\frac{f\left( a,c\right) +f\left( a,d\right) +f\left(
b,c\right) +f\left( b,d\right) }{4}\right. \\
&&+\frac{\lambda \left( 1-\lambda \right) }{2}\left[ f\left( \frac{a+b}{2}%
,c\right) +f\left( \frac{a+b}{2},d\right) +f\left( a,\frac{c+d}{2}\right)
+f\left( b,\frac{c+d}{2}\right) \right] \\
&&-\left( 1-\lambda \right) \frac{1}{d-c}\int_{c}^{d}f(\frac{a+b}{2}%
,y)dy-\left( 1-\lambda \right) \frac{1}{b-a}\int_{a}^{b}f(x,\frac{c+d}{2})dx
\\
&&-\lambda \frac{1}{2\left( b-a\right) }\int_{a}^{b}\left[ f(x,d)+f(x,c)%
\right] dx-\lambda \frac{1}{2\left( d-c\right) }\int_{c}^{d}\left[
f(a,y)+f(b,y)\right] dy \\
&&\left. +\frac{1}{(b-a)(d-c)}\int_{a}^{b}\int_{c}^{d}f(x,y)dxdy\right\vert
\\
&\leq &\frac{1}{(b-a)(d-c)}\left( \int_{a}^{b}\int_{c}^{d}\left\vert
K(x)M(y)\right\vert ^{p}dydx\right) ^{\frac{1}{p}}\left(
\int_{a}^{b}\int_{c}^{d}\left\vert \frac{\partial ^{2}f}{\partial x\partial y%
}(x,y)\right\vert ^{q}dydx\right) ^{\frac{1}{q}}.
\end{eqnarray*}%
Since $\left\vert \frac{\partial ^{2}f}{\partial x\partial y}\right\vert ^{%
\frac{p}{p-1}}$ is convex function on the co-ordinates on $\Delta ,$ by
taking into account the change of variable $x=tb+(1-t)a,$ $(b-a)dt=dt$ and $%
y=sd+(1-s)c,$ $(d-c)ds=dy,$ we have%
\begin{equation*}
\left\vert \frac{\partial ^{2}f}{\partial x\partial y}\left( tb+\left(
1-t\right) a,y\right) \right\vert ^{q}\leq t\left\vert \frac{\partial ^{2}f}{%
\partial x\partial y}\left( b,y\right) \right\vert ^{q}+\left( 1-t\right)
\left\vert \frac{\partial ^{2}f}{\partial x\partial y}\left( a,y\right)
\right\vert ^{q}
\end{equation*}%
and%
\begin{eqnarray*}
&&\left\vert \frac{\partial ^{2}f}{\partial x\partial y}\left( tb+\left(
1-t\right) a,sd+\left( 1-s\right) c\right) \right\vert ^{q} \\
&\leq &ts\left\vert \frac{\partial ^{2}f}{\partial x\partial y}\right\vert
^{q}(b,d)+t\left( 1-s\right) \left\vert \frac{\partial ^{2}f}{\partial
x\partial y}\right\vert ^{q}(b,c) \\
&&+s\left( 1-t\right) \left\vert \frac{\partial ^{2}f}{\partial x\partial y}%
\right\vert ^{q}(a,d)+\left( 1-t\right) (1-s)\left\vert \frac{\partial ^{2}f%
}{\partial x\partial y}\right\vert ^{q}(a,c)
\end{eqnarray*}%
thus, we obtain%
\begin{eqnarray*}
&&\left\vert \left( 1-\lambda \right) ^{2}f\left( \frac{a+b}{2},\frac{c+d}{2}%
\right) +\lambda ^{2}\frac{f\left( a,c\right) +f\left( a,d\right) +f\left(
b,c\right) +f\left( b,d\right) }{4}\right. \\
&&+\frac{\lambda \left( 1-\lambda \right) }{2}\left[ f\left( \frac{a+b}{2}%
,c\right) +f\left( \frac{a+b}{2},d\right) +f\left( a,\frac{c+d}{2}\right)
+f\left( b,\frac{c+d}{2}\right) \right] \\
&&-\left( 1-\lambda \right) \frac{1}{d-c}\int_{c}^{d}f(\frac{a+b}{2}%
,y)dy-\left( 1-\lambda \right) \frac{1}{b-a}\int_{a}^{b}f(x,\frac{c+d}{2})dx
\\
&&-\lambda \frac{1}{2\left( b-a\right) }\int_{a}^{b}\left[ f(x,d)+f(x,c)%
\right] dx-\lambda \frac{1}{2\left( d-c\right) }\int_{c}^{d}\left[
f(a,y)+f(b,y)\right] dy \\
&&\left. +\frac{1}{(b-a)(d-c)}\int_{a}^{b}\int_{c}^{d}f(x,y)dxdy\right\vert
\\
&\leq &\frac{(b-a)(d-c)}{4\left( p+1\right) ^{\frac{2}{p}}}\left[ 2\lambda
^{2}-2\lambda +1\right] ^{2} \\
&&\times \left( \frac{\left\vert \frac{\partial ^{2}f}{\partial x\partial y}%
\right\vert ^{q}(a,c)+\left\vert \frac{\partial ^{2}f}{\partial x\partial y}%
\right\vert ^{q}(a,d)+\left\vert \frac{\partial ^{2}f}{\partial x\partial y}%
\right\vert ^{q}(b,c)+\left\vert \frac{\partial ^{2}f}{\partial x\partial y}%
\right\vert ^{q}(b,d)}{4}\right) ^{\frac{1}{q}}.
\end{eqnarray*}%
Which completes the proof.
\end{proof}

\begin{remark}
Under the assumptions of Theorem 6, if we choose $\lambda =1,$ we have the
inequality (\ref{1.4}).
\end{remark}

\begin{remark}
Under the assumptions of Theorem 6, if we choose $\lambda =0,$ we have the
inequality (\ref{1.7}).
\end{remark}

\begin{corollary}
\label{cet}Under the assumptions of Theorem 6, if we choose $\lambda =\frac{1%
}{3},$ we have the inequality%
\begin{eqnarray*}
&&\left\vert \frac{f\left( a,\frac{c+d}{2}\right) +f\left( b,\frac{c+d}{2}%
\right) +4f\left( \frac{a+b}{2},\frac{c+d}{2}\right) +f\left( \frac{a+b}{2}%
,c\right) +f\left( \frac{a+b}{2},d\right) }{9}\right. \\
&&\left. +\frac{f\left( a,c\right) +f\left( b,c\right) +f\left( a,d\right)
+f\left( b,d\right) }{36}+\frac{1}{\left( b-a\right) \left( d-c\right) }%
\int_{a}^{b}\int_{c}^{d}f\left( x,y\right) dydx-A\right\vert \\
&\leq &\frac{25\left( b-a\right) \left( d-c\right) }{324\left( p+1\right) ^{%
\frac{2}{p}}} \\
&&\times \left( \frac{\left\vert \frac{\partial ^{2}f}{\partial x\partial y}%
\right\vert ^{q}(a,c)+\left\vert \frac{\partial ^{2}f}{\partial x\partial y}%
\right\vert ^{q}(a,d)+\left\vert \frac{\partial ^{2}f}{\partial x\partial y}%
\right\vert ^{q}(b,c)+\left\vert \frac{\partial ^{2}f}{\partial x\partial y}%
\right\vert ^{q}(b,d)}{4}\right) ^{\frac{1}{q}}
\end{eqnarray*}%
where%
\begin{eqnarray*}
A &=&\frac{1}{6\left( b-a\right) }\int_{a}^{b}\left[ f\left( x,c\right)
+4f\left( x,\frac{c+d}{2}\right) +f\left( x,d\right) \right] dx \\
&&+\frac{1}{6\left( d-c\right) }\int_{c}^{d}\left[ f\left( a,y\right)
+4f\left( \frac{a+b}{2},y\right) +f\left( b,y\right) \right] dy.
\end{eqnarray*}
\end{corollary}

\begin{corollary}
In Corollary \ref{cet}, since $\frac{1}{4}<\frac{1}{(p+1)^{\frac{2}{p}}}<1,$
for $p>1$, we have the following inequality;%
\begin{eqnarray*}
&&\left\vert \frac{f\left( a,\frac{c+d}{2}\right) +f\left( b,\frac{c+d}{2}%
\right) +4f\left( \frac{a+b}{2},\frac{c+d}{2}\right) +f\left( \frac{a+b}{2}%
,c\right) +f\left( \frac{a+b}{2},d\right) }{9}\right. \\
&&\left. +\frac{f\left( a,c\right) +f\left( b,c\right) +f\left( a,d\right)
+f\left( b,d\right) }{36}+\frac{1}{\left( b-a\right) \left( d-c\right) }%
\int_{a}^{b}\int_{c}^{d}f\left( x,y\right) dydx-A\right\vert \\
&\leq &\frac{25\left( b-a\right) \left( d-c\right) }{324} \\
&&\times \left( \frac{\left\vert \frac{\partial ^{2}f}{\partial x\partial y}%
\right\vert ^{q}(a,c)+\left\vert \frac{\partial ^{2}f}{\partial x\partial y}%
\right\vert ^{q}(a,d)+\left\vert \frac{\partial ^{2}f}{\partial x\partial y}%
\right\vert ^{q}(b,c)+\left\vert \frac{\partial ^{2}f}{\partial x\partial y}%
\right\vert ^{q}(b,d)}{4}\right) ^{\frac{1}{q}}.
\end{eqnarray*}
\end{corollary}

\begin{theorem}
Let $f:\Delta =[a,b]\times \lbrack c,d]\rightarrow 
\mathbb{R}
$ be a differentiable function on $\Delta $. If $\left\vert \frac{\partial
^{2}f}{\partial x\partial y}\right\vert ^{q}$ is convex function on the
co-ordinates on $\Delta $ and $q\geq 1,$ then one has the inequality:%
\begin{eqnarray*}
&&\left\vert \left( 1-\lambda \right) ^{2}f\left( \frac{a+b}{2},\frac{c+d}{2}%
\right) +\lambda ^{2}\frac{f\left( a,c\right) +f\left( a,d\right) +f\left(
b,c\right) +f\left( b,d\right) }{4}\right. \\
&&+\frac{\lambda \left( 1-\lambda \right) }{2}\left[ f\left( \frac{a+b}{2}%
,c\right) +f\left( \frac{a+b}{2},d\right) +f\left( a,\frac{c+d}{2}\right)
+f\left( b,\frac{c+d}{2}\right) \right] \\
&&-\left( 1-\lambda \right) \frac{1}{d-c}\int_{c}^{d}f(\frac{a+b}{2}%
,y)dy-\left( 1-\lambda \right) \frac{1}{b-a}\int_{a}^{b}f(x,\frac{c+d}{2})dx
\\
&&-\lambda \frac{1}{2\left( b-a\right) }\int_{a}^{b}\left[ f(x,d)+f(x,c)%
\right] dx-\lambda \frac{1}{2\left( d-c\right) }\int_{c}^{d}\left[
f(a,y)+f(b,y)\right] dy \\
&&\left. +\frac{1}{(b-a)(d-c)}\int_{a}^{b}\int_{c}^{d}f(x,y)dxdy\right\vert
\\
&\leq &\frac{(b-a)(d-c)}{16}\left[ 2\lambda ^{2}-2\lambda +1\right] ^{2} \\
&&\times \left( \frac{\left\vert \frac{\partial ^{2}f}{\partial x\partial y}%
\right\vert ^{q}(a,c)+\left\vert \frac{\partial ^{2}f}{\partial x\partial y}%
\right\vert ^{q}(a,d)+\left\vert \frac{\partial ^{2}f}{\partial x\partial y}%
\right\vert ^{q}(b,c)+\left\vert \frac{\partial ^{2}f}{\partial x\partial y}%
\right\vert ^{q}(b,d)}{4}\right) ^{\frac{1}{q}}.
\end{eqnarray*}
\end{theorem}

\begin{proof}
From Lemma 1 and using the well-known Power-mean inequality, we can write%
\begin{eqnarray*}
&&\left\vert \left( 1-\lambda \right) ^{2}f\left( \frac{a+b}{2},\frac{c+d}{2}%
\right) +\lambda ^{2}\frac{f\left( a,c\right) +f\left( a,d\right) +f\left(
b,c\right) +f\left( b,d\right) }{4}\right. \\
&&+\frac{\lambda \left( 1-\lambda \right) }{2}\left[ f\left( \frac{a+b}{2}%
,c\right) +f\left( \frac{a+b}{2},d\right) +f\left( a,\frac{c+d}{2}\right)
+f\left( b,\frac{c+d}{2}\right) \right] \\
&&-\left( 1-\lambda \right) \frac{1}{d-c}\int_{c}^{d}f(\frac{a+b}{2}%
,y)dy-\left( 1-\lambda \right) \frac{1}{b-a}\int_{a}^{b}f(x,\frac{c+d}{2})dx
\\
&&-\lambda \frac{1}{2\left( b-a\right) }\int_{a}^{b}\left[ f(x,d)+f(x,c)%
\right] dx-\lambda \frac{1}{2\left( d-c\right) }\int_{c}^{d}\left[
f(a,y)+f(b,y)\right] dy \\
&&\left. +\frac{1}{(b-a)(d-c)}\int_{a}^{b}\int_{c}^{d}f(x,y)dxdy\right\vert
\\
&\leq &\frac{1}{(b-a)(d-c)}\left( \int_{a}^{b}\int_{c}^{d}\left\vert
K(x)M(y)\right\vert dydx\right) ^{1-\frac{1}{q}} \\
&&\times \left( \int_{a}^{b}\int_{c}^{d}\left\vert K(x)M(y)\right\vert
\left\vert \frac{\partial ^{2}f}{\partial x\partial y}(x,y)\right\vert
^{q}dydx\right) ^{\frac{1}{q}}.
\end{eqnarray*}%
Since $\left\vert \frac{\partial ^{2}f}{\partial x\partial y}\right\vert
^{q} $ is convex function on the co-ordinates on $\Delta ,$ by taking into
account the change of variable $x=tb+(1-t)a,$ $(b-a)dt=dt$ and $y=sd+(1-s)c,$
$(d-c)ds=dy,$ we have%
\begin{equation*}
\left\vert \frac{\partial ^{2}f}{\partial x\partial y}\left( tb+\left(
1-t\right) a,y\right) \right\vert ^{q}\leq t\left\vert \frac{\partial ^{2}f}{%
\partial x\partial y}\left( b,y\right) \right\vert ^{q}+\left( 1-t\right)
\left\vert \frac{\partial ^{2}f}{\partial x\partial y}\left( a,y\right)
\right\vert ^{q}
\end{equation*}%
and%
\begin{eqnarray*}
&&\left\vert \frac{\partial ^{2}f}{\partial x\partial y}\left( tb+\left(
1-t\right) a,sd+\left( 1-s\right) c\right) \right\vert ^{q} \\
&\leq &ts\left\vert \frac{\partial ^{2}f}{\partial x\partial y}\right\vert
^{q}(b,d)+t\left( 1-s\right) \left\vert \frac{\partial ^{2}f}{\partial
x\partial y}\right\vert ^{q}(b,c) \\
&&+s\left( 1-t\right) \left\vert \frac{\partial ^{2}f}{\partial x\partial y}%
\right\vert ^{q}(a,d)+\left( 1-t\right) (1-s)\left\vert \frac{\partial ^{2}f%
}{\partial x\partial y}\right\vert ^{q}(a,c)
\end{eqnarray*}%
hence, it follows that%
\begin{eqnarray*}
&&\left\vert \left( 1-\lambda \right) ^{2}f\left( \frac{a+b}{2},\frac{c+d}{2}%
\right) +\lambda ^{2}\frac{f\left( a,c\right) +f\left( a,d\right) +f\left(
b,c\right) +f\left( b,d\right) }{4}\right. \\
&&+\frac{\lambda \left( 1-\lambda \right) }{2}\left[ f\left( \frac{a+b}{2}%
,c\right) +f\left( \frac{a+b}{2},d\right) +f\left( a,\frac{c+d}{2}\right)
+f\left( b,\frac{c+d}{2}\right) \right] \\
&&-\left( 1-\lambda \right) \frac{1}{d-c}\int_{c}^{d}f(\frac{a+b}{2}%
,y)dy-\left( 1-\lambda \right) \frac{1}{b-a}\int_{a}^{b}f(x,\frac{c+d}{2})dx
\\
&&-\lambda \frac{1}{2\left( b-a\right) }\int_{a}^{b}\left[ f(x,d)+f(x,c)%
\right] dx-\lambda \frac{1}{2\left( d-c\right) }\int_{c}^{d}\left[
f(a,y)+f(b,y)\right] dy \\
&&\left. +\frac{1}{(b-a)(d-c)}\int_{a}^{b}\int_{c}^{d}f(x,y)dxdy\right\vert
\\
&\leq &\frac{(b-a)(d-c)}{16}\left[ 2\lambda ^{2}-2\lambda +1\right] ^{2} \\
&&\times \left( \frac{\left\vert \frac{\partial ^{2}f}{\partial x\partial y}%
\right\vert ^{q}(a,c)+\left\vert \frac{\partial ^{2}f}{\partial x\partial y}%
\right\vert ^{q}(a,d)+\left\vert \frac{\partial ^{2}f}{\partial x\partial y}%
\right\vert ^{q}(b,c)+\left\vert \frac{\partial ^{2}f}{\partial x\partial y}%
\right\vert ^{q}(b,d)}{4}\right) ^{\frac{1}{q}}.
\end{eqnarray*}%
Which completes the proof.
\end{proof}

\begin{remark}
Under the assumptions of Theorem 7, if we choose $\lambda =1,$ we have the
inequality (\ref{1.5}).
\end{remark}

\begin{remark}
Under the assumptions of Theorem 7, if we choose $\lambda =0,$ we have the
inequality (\ref{1.8}).
\end{remark}

\begin{corollary}
Under the assumptions of Theorem 7, if we choose $\lambda =\frac{1}{3},$ we
have%
\begin{eqnarray*}
&&\left\vert \frac{f\left( a,\frac{c+d}{2}\right) +f\left( b,\frac{c+d}{2}%
\right) +4f\left( \frac{a+b}{2},\frac{c+d}{2}\right) +f\left( \frac{a+b}{2}%
,c\right) +f\left( \frac{a+b}{2},d\right) }{9}\right. \\
&&\left. +\frac{f\left( a,c\right) +f\left( b,c\right) +f\left( a,d\right)
+f\left( b,d\right) }{36}+\frac{1}{\left( b-a\right) \left( d-c\right) }%
\int_{a}^{b}\int_{c}^{d}f\left( x,y\right) dydx-A\right\vert \\
&\leq &\frac{25\left( b-a\right) \left( d-c\right) }{36^{2}} \\
&&\times \left( \frac{\left\vert \frac{\partial ^{2}f}{\partial x\partial y}%
\right\vert ^{q}(a,c)+\left\vert \frac{\partial ^{2}f}{\partial x\partial y}%
\right\vert ^{q}(a,d)+\left\vert \frac{\partial ^{2}f}{\partial x\partial y}%
\right\vert ^{q}(b,c)+\left\vert \frac{\partial ^{2}f}{\partial x\partial y}%
\right\vert ^{q}(b,d)}{4}\right) ^{\frac{1}{q}}.
\end{eqnarray*}%
where%
\begin{eqnarray*}
A &=&\frac{1}{6\left( b-a\right) }\int_{a}^{b}\left[ f\left( x,c\right)
+4f\left( x,\frac{c+d}{2}\right) +f\left( x,d\right) \right] dx \\
&&+\frac{1}{6\left( d-c\right) }\int_{c}^{d}\left[ f\left( a,y\right)
+4f\left( \frac{a+b}{2},y\right) +f\left( b,y\right) \right] dy.
\end{eqnarray*}
\end{corollary}


\begin{thebibliography}{9}
\bibitem{BAK} M.K. Bakula and J. Pe\v{c}ari\'{c}, On the Jensen's inequality
for convex functions on the co-ordinates in a rectangle from the plane, 
\textit{Taiwanese Journal of Math.}, 5, 2006, 1271-1292.

\bibitem{OZ} M.E. \"{O}zdemir, E. Set, M.Z. Sar\i kaya, Some new Hadamard's
type inequalities for co-ordinated $m-$convex and ($\alpha ,m)-$convex
functions, Accepted.

\bibitem{OZ2} M.Z. Sar\i kaya, E. Set, M. Emin \"{O}zdemir and S.S.
Dragomir, New some Hadamard's type inequalities for co-ordinated convex
functions, Accepted.

\bibitem{SS} S.S. Dragomir, On Hadamard's inequality for convex functions on
the co-ordinates in a rectangle from the plane, \textit{Taiwanese Journal of
Math.}, 5, 2001, 775-788.

\bibitem{HTS} D. Y. Hwang, K. L. Tseng and G. S. Yang, Some Hadamard's
inequalities for co-ordinated convex functions in a rectangle from the
plane, \textit{Taiwanese Journal of Mathematics,} 11 (2007), 63-73.

\bibitem{OZ3} M.E. \"{O}zdemir, H. Kavurmac\i , A.O. Akdemir and M. Avc\i ,
Inequalities for convex and $s-$convex functions on $\Delta =\left[ a,b%
\right] \times \left[ c,d\right] ,$ Submitted.

\bibitem{OZ4} M.E. \"{O}zdemir, A.O. Akdemir and H. Kavurmac\i , On the
Simpson's inequality for co-ordinated convex functions$,$ Submitted.
\end{thebibliography}
\end{document}